\documentclass[a4paper,twoside,11pt]{article}
\usepackage{amssymb}
\usepackage{amsmath}
\usepackage{graphicx}
\usepackage{amsthm}
\usepackage{cite}
\usepackage{array}

\usepackage{amsmath,latexsym,amssymb,amsfonts,amsbsy,amsthm}
\usepackage{graphicx}
\usepackage{subfigure}
\usepackage[top=1in, bottom=1in, left=1.25in, right=1.25in]
{geometry}
\usepackage{epstopdf}
\usepackage{diagbox}
\usepackage{enumerate}
\usepackage{color}
\usepackage{comment}  
\usepackage{indentfirst} 
\newfont{\bb}{msbm10}

\newtheorem{theorem}{Theorem}[section]

\newtheorem{lemma}{Lemma}[section]

 \usepackage{verbatim}
 \usepackage{algorithm}
 \usepackage{algorithmicx}
 \usepackage{algpseudocode}
 \usepackage{amsmath}
 \usepackage{graphics}
 \usepackage{epsfig}
 \usepackage{arydshln}

 \usepackage{mathrsfs}
 \usepackage{graphicx}
 \usepackage{subfigure}
  \usepackage{caption}
  \usepackage{multirow} 
  \usepackage{threeparttable}
  \numberwithin{equation}{section}

  \usepackage{footnote}
  \makesavenoteenv{table}
\usepackage{authblk}
  \usepackage{hyperref}
\hypersetup{
hidelinks,
colorlinks=true,
linkcolor=blue,
citecolor=blue
}

 \title{Restarted randomized surrounding methods for solving large linear equations}

 \author{
	Jun-Feng Yin\\
	School of Mathematical Sciences, Tongji University,\\
	Shanghai, 200092, PR China.\\
	Email:yinjf@tongji.edu.cn\\
    Nan Li\\
	School of Mathematical Sciences, Tongji University,\\
	Shanghai, 200092, PR China.\\
	Email:sunshinekiwili@tongji.edu.cn\\
	and\\
	Ning Zheng\\
	School of Mathematical Sciences, Tongji University,\\
	Shanghai, 200092, PR China.\\
	Email: nzheng@tongji.edu.cn\\
}
\begin{document}
\cleardoublepage \pagestyle{myheadings}



\markboth{\small J.-F. Yin, N. Li and N. Zheng  }
{\small Restarted randomized surrounding methods for solving large linear equations }

 \captionsetup[figure]{labelfont={bf},labelformat={default},labelsep=period,name={Fig.}}	
  \captionsetup[table]{labelfont={bf},labelformat={default},labelsep=period,name={Table }}		\maketitle
  \date{ }
\begin{abstract}
A class of restarted randomized surrounding methods are presented to accelerate the surrounding algorithms by restarted techniques for solving the linear equations.
Theoretical analysis shows that the proposed method converges under the randomized row selection rule and the convergence rate in expectation is also addressed.
Numerical experiments further demonstrate that the proposed algorithms are efficient and outperform the existing method for overdetermined and underdetermined linear equations, as well as in the application of image processing.
\end{abstract}


\section{Introduction}
	Consider the solution of linear algebraic equations
	\begin{equation}\label{eqn:axb}
Ax=b,\quad A\in\mathbb{R}^{m\times n},\quad b\in\mathbb{R}^{m},
	\end{equation}
where $A$ has full column rank, which comes widely from many scientific and engineering computation, for instance, discrete PDEs, image reconstruction, signal processing, option pricing and machine learning.

Kaczmarz method is one of the well-known iterative projection method,
which was firstly proposed in \cite{1937K} and further extended to block and inconsistent cases in \cite{1980E,1981C}.
Since the linear convergence of a randomized Kaczmarz method was established  by Strohmer and Vershynin \cite{2009SV},  variants of randomized Kaczmarz method were presented and deeply studied, see \cite{2018BWgreedy,2019BW,2022JZY,2021Du}. On the other hand, iterative methods based on Householder orthogonal reflection also attract much attention from the community of numerical linear algebra.
Cimmino \cite{1938cimmino} firstly proposed a general iteration scheme with orthogonal reflection and proved the convergence as long as the
rank of $A$ is greater than one. Ansorge \cite{1984Ansorge} gave the relations between the Cimmino method and Kaczmarz method for the solution of singular and rectangular systems of equations and proved the rate of convergence for a given weight.
Further, block Cimmino method\cite{1995ADRS,2018TMA} and extended Cimmino method \cite{2008PPS} were proposed and investigated.
For more details of the Cimmino method, we refer the reader to \cite{2004Benzi}.

Recently, Steinerberger\cite{2021Steinerberger} studied a surrounding method
which randomly reflected the start point and took the average of all reflective points
as the approximate solution.
In this manuscript, we propose a restarted randomized surrounding method, which takes the average of several randomly reflective points as the new initial value and repeats the iterations.
Theoretical analysis demonstrates the convergence in expectation
and shows the convergence rate is faster than that of the existing surrounding method.
Numerical experiments further verify our analysis, and show that restarted strategies are efficient which can greatly accelerate the surrounding method.

The organization of the rest paper is as follows. In Section 2, we give some notations and propose the restarted randomized surrounding algorithms. In Section 3, the convergence analysis is given and compared with the existing results. Numerical experiments are presented in Section 4 compared with the surrounding method. Finally, in Section 5, we end this paper with the conclusions.

\section{The restarted randomized surrounding method}

In this section, after introducing  the notations and reviewing the existing method, we introduce the restarted randomized surrounding method for solving the linear equations \eqref{eqn:axb}.

Denote $a_i^T=[a_{i1},a_{i2},\ldots,a_{in}]$ and $b_i$  be the $i$th row of $A$ and the $i$th entry of the right-hand side vector $b$ respectively. Let $x_k$ be the $k$th approximate solution and $x_\ast$ be the exact solution of the linear equations \eqref{eqn:axb} respectively, which is actually the intersection of the $n$ hyperplanes
$a_i^Tx=b_i  (1\leq i \leq m)$.

By choosing $i_k$ from the set $\left\{1,2,\cdots, m \right\} $ with probability proportional to $ \| a_i\|_2^2$, Steinerberger \cite{2021Steinerberger} proposed a surrounding method as follows.
\begin{eqnarray}
x_{k+1}  &=& x_k+2\frac{b_{i_k}-a_{i_k}^Tx_k}{\|a_{i_k}\|_2^2}a_{i_k}, \quad 1\leq k \leq M, \\
   &=& (I - 2\frac{a_{i_k}^Ta_{i_k}}{\|a_{i_k}\|_2^2})x_k+2\frac{b_{i_k}}{\|a_{i_k}\|_2^2}a_{i_k}, \quad 1\leq k \leq M.
\end{eqnarray}
After $M$ reflections,
the approximate solution is given by the average of all the reflective points $ \frac{1}{M} \sum_{k=1}^M x_k$.

Steinerberger \cite{2021Steinerberger} proved that the approximate solution approaches to  the true solution $x_\ast$ when $M$ goes to infinity, e.g.,
$\lim\limits_{M \rightarrow \infty} \frac{1}{M}\sum_{i=1}^{M} x_k= x_\ast$ and the convergence rate in expectation was given by \begin{equation}
\mathbb{E}\left\|x_\ast-\frac{1}{M} \sum_{k=1}^{M} x_{k}\right\| \leq \frac{1+\|A\|_{F}\left\|A^{-1}\right\|}{\sqrt{M}} \left\|x_\ast-x_{0}\right\|.
\end{equation}

It is easily seen that it requires quite a number of reflective points to achieve a satisfying convergence precision, which is usually very expensive.

In order to accelerate the convergence of the randomized surrounding method,
 a restarted version is proposed which firstly reflects several times and  then  takes the average of the reflective points as the initial to restart the iterations.

The framework of the restarted randomized surrounding method (abbreviated as RRS) can be described as follow:

\begin{algorithm} [H]
	\caption{Restarted randomized surrounding (RRS) algorithms}
	\label{alg2}
	\begin{algorithmic}[1]
		\Require  $A$, $b$, initial guess $x_{0}=y_{0}^{(0)}\in\mathbb{R}^{n}$, restarted number $q$.
		\Ensure approximate solution $x_{k+1}$.
		\For {$k=0,1,2,\ldots$ }
		\For{ $i = 1,2,\ldots,q-1$}
		\State Select a row index $i_k$ randomly according to $p_{i_k}=\frac{\|a_{i_k}\|^2}{\|A\|_F^2}$
		\State Set $y_k^{(i)}=y_k^{(i-1)}+2\frac{b_{i_k}-a^T_{i_k}y_{k}^{(i-1)}}{\|a_{i_k}\|_2^2}a_{i_k}$
		\EndFor
		\State Compute the approximate solution:  $x_{k+1}=\frac{1}{q}\sum_{i=0}^{q-1}y_{k}^{(i)}$
		\State Update $y_{k+1}^{(0)}=x_{k+1}$
		\EndFor\\
		\Return $x_{k+1}$
	\end{algorithmic}
\end{algorithm}

We plot a sketch graph to demonstrate the idea of restarted randomized surrounding method  in Figure \ref{fig:sketch2},
where the green diamond points are the approximate solutions which eventually converge to the center, namely the true solution.
In Figure \ref{fig:sketch2}, the star points and circle points denote the randomized reflective points generated in the first and second loops respectively.

\begin{figure}[!htbp]
	\centering
	\includegraphics[width=0.7\linewidth]{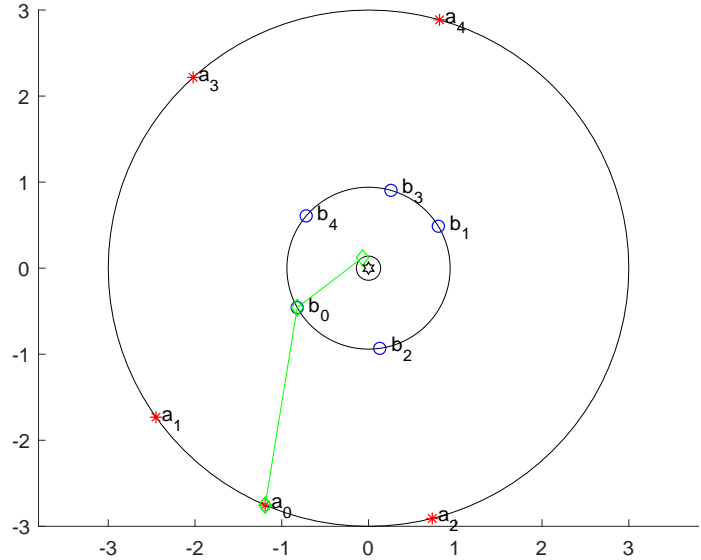}
	\caption{Sketch of restarted randomized surrounding method}
	\label{fig:sketch2}
\end{figure}

\section{Convergence analysis}
Denote the Householder reflection matrix corresponding to the $i$th hyperplane by
\begin{equation}
H_i=I-2\frac{a_{i}a_{i}^T}{a_{i}^Ta_{i}},\quad i=1,2,\ldots,m.
\end{equation}
which is an $n\times n$ orthogonal matrix.
We give the convergence of restarted randomized surrounding method as follows.

\begin{lemma}\label{mythm1}
In every iteration, if $q$ reflections are randomly computed where $i_k$ is taken  from the set $\left\{1,2,\cdots, m \right\} $ with the probability $p_{i_k}=\frac{\|a_{i_k}\|^2}{\|A\|_F^2}$, then the sequence $\{\|x_k-x_\ast\|\}$ is non-increasing.
\end{lemma}
\begin{proof}
	Without loss of generality, we assume that the rows with indices $\{i_{1},\ldots,i_{q}\}$ are selected in $k$-th iteration step and set
	\begin{equation}
	Q_{j}=H_{i_j}H_{i_{j-1}}\cdots H_{i_1},\,Q_0=I,\,j=1,\ldots,q.
	\end{equation}
	where $I$ represents the identity matrix of size $n$. Since $Q_{j}$ is a product of the Householder matrices, $Q_{j}$ is orthogonal and nonsingular.
	Hence,
	\begin{equation}
	\begin{aligned}
	x_{k+1}-x_\ast&=\frac{1}{q}\left(\sum_{i=1}^{q}y_k^{(i)}\right)-x_\ast\\
&=\frac{1}{q}\sum_{i=1}^{q}(y_k^{(i)}-x_\ast)\\
	&=\frac{1}{q}(Q_1+Q_2+\cdots+Q_q)( x_{k}-x_\ast). 
	\end{aligned}
	\end{equation}
	Due to distance-preserving transformation,
	\begin{equation}\label{equalcondition}
	\| x_{k+1}-x_\ast\|\leq \frac{1}{q} \sum_{j=1}^{q}\|Q_{j}(x_k-x_\ast)\|= \frac{1}{q}\sum_{i=1}^{q}\|(x_k-x_\ast)\|=\|(x_k-x_\ast)\|.
	\end{equation}

	Because the linear combination is convex, the equality in \eqref{equalcondition}  holds if and only if
$$x_{k}-x_\ast=Q_1 (x_{k}-x_\ast)=\cdots=Q_q(x_{k}-x_\ast).$$
 If $x_k\not=x_\ast$ and $\mathcal{N}(A)=\{0\}$, then $\|\frac{1}{q}(Q_1+Q_2+\cdots+Q_q)\|<1$, which yields the sequence $\{\|x_k-x_\ast\|\}$ is decreasing.
\end{proof}

From the Lemma \ref{mythm1}, it is seen that the restarted randomized  surrounding  algorithm is convergent.
Further, we estimate and analyze the convergence rate of the restarted randomized surrounding algorithm as follow.

\begin{theorem}\label{th:con}
If $i_k$ is taken with the probability $p_{i_k}=\frac{\|a_{i_k}\|^2}{\|A\|_F^2}$
and  $q$ reflections are  randomly taken in every iteration, it holds
\begin{equation}
\mathbb{E}\|x_\ast-x_{k}\|^2\leq \gamma^{k} \|x_0-x_\ast\|^2, \quad 0<\gamma<1,
\end{equation}
where the constant $\gamma<\frac{1}{q}+\frac{2}{q^2}\sum_{i=1}^{q-1}(q-i)\|\mathcal{L}\|^i$ and $\|\mathcal{L}\|=1-\frac{2\sigma_{\min}^2}{\|A\|_F^2}$.
\end{theorem}
\begin{proof}	
Denote $\langle \cdot,\cdot\rangle$ be the Euclidean inner product,
for all $ x,y\in\mathbb{R}^{n}$, $i_j\in\{1,2,\ldots,m\}$,
we have
	\begin{equation}
	\begin{aligned}
	&\mathbb{E}_{i_j}\left(\langle y, Q_{j}x\rangle\right)
	=\mathbb{E}_{i_j}\left(\langle y,(I-2\frac{a_{i_j}a_{i_j}^T}{\|a_{i_j}\|^2})Q_{j-1}x\rangle\right)\\
	=&\langle y,Q_{j-1}x\rangle-2\mathbb{E}_{i_j}\left(\langle y,\frac{a_{i_j}a_{i_j}^T}{\|a_{i_j}\|^2}Q_{j-1}x\rangle\right)\\
	=&\langle y,Q_{j-1}x\rangle-2\langle y,\sum_{j=1}^{m}\frac{\|a_{i_j}\|^2}{\|A\|_F^2}\left(\frac{a_{i_j}a_{i_j}^T}{\|a_{i_j}\|^2}Q_{j-1}x\right)\rangle\\
	=&\langle y,Q_{j-1}x\rangle-\frac{2}{\|A\|_F^2} \sum_{j=1}^{m}y^Ta_{i_j}a_{i_j}^T
Q_{j-1} x\\
	=&\langle (I-\frac{2}{\|A\|_F^2}A^TA)y,Q_{j-1}x\rangle:=\langle \mathcal{L}y,Q_{j-1}x\rangle.
	\end{aligned}
	\end{equation}	
	
	Similarly, it is obtained that
	\begin{equation}
	\mathbb{E}_{i_1,i_2,\ldots,i_{j}}\langle y,Q_j x\rangle=\mathbb{E}_{i_1,i_2,\ldots,i_{j-1}}\langle \mathcal{L}y,Q_{j-1}x\rangle=\mathbb{E}_{i_1,i_2,\ldots,i_{j-2}}\langle \mathcal{L}^2y,Q_{j-2}x\rangle=\cdots=\langle \mathcal{L}^jy,x\rangle.
	\end{equation}
	Since $\mathcal{L}= I - \frac{2}{\|A\|_F^2}A^TA $ is a symmetric matrix,
	\begin{equation}
	\begin{aligned}
	\rho(\mathcal{L})=\|\mathcal{L}\| \leq 1-\frac{2\sigma_{\min}^2}{\|A\|_F^2},
	\end{aligned}
	\end{equation}
	where $\sigma_{\min}$ is the smallest singular value of $A$. Thus,
	$$
	|\mathbb{E}_{i_1,i_2,\ldots,i_{j}}\langle x,Q_j x\rangle|\leq \|\mathcal{L}^jx\|\|x\|\leq \|\mathcal{L}\|^j\|x\|^2.$$

	Let $e_k=x_k-x_\ast$ and it holds that
	\begin{equation}
	\begin{aligned}
	&\mathbb{E}\|x_\ast-x_{k+1}\|=\mathbb{E}\|x_\ast-\frac{1}{q}\sum_{j=0}^{q-1}Q_{j}x_k\|^2=\frac{1}{q^2}\mathbb{E}\|\sum_{j=0}^{q-1} Q_{j}e_k\|^2\\
	&=\frac{1}{q^2}\left[\sum_{j=0}^{q-1}\langle Q_{j} e_k,Q_{j} e_k\rangle+2\mathbb{E}\sum_{j=0}^{q-2}\sum_{\ell=j+1}^{q-1}\langle Q_{j}e_k,Q_{\ell}e_{k}\rangle\right]\\
	&=\frac{1}{q}\|e_k\|^2+\frac{2}{q^2}\mathbb{E}\sum_{j=0}^{q-2}\sum_{\ell=j+1}^{q-1}\langle Q_{j}e_k, H_{i_\ell}\cdots H_{i_{j+1}}(Q_{j}e_k)\rangle\\
	&\leq \frac{1}{q}\|e_k\|^2+\frac{2}{q^2}\sum_{j=0}^{q-2}\sum_{\ell=j+1}^{q-1}\|\mathcal{L}^{\ell-j}\|\|Q_{j}e_k\|^2:=\gamma\|e_k\|^2,
	\end{aligned}
	\end{equation}where $\gamma$ is a constant and $\|\mathcal{L}\|<1$. By calculating, $$\begin{aligned}
	\gamma&=\frac{1}{q}+\frac{2}{q^2}\left[(q-1)\|\mathcal{L}\|+(q-2)\|\mathcal{L}^2\|+\cdots+\|\mathcal{L}^{q-1}\|\right]\\
	&\leq\frac{1}{q}+\frac{2}{q^2}\left[ (q-1)\|\mathcal{L}\|+(q-2)\|\mathcal{L}\|^2+\cdots+\|\mathcal{L}\|^{q-1}\right]<1.\\
	\end{aligned}$$

Hence, it holds that
	$$\mathbb{E}\|x_\ast-x_{k+1}\|^2\leq \gamma\|x_{k}-x_\ast\|^2\Rightarrow\mathbb{E}\|x_\ast-x_{k}\|^2\leq \gamma^k \|x_0-x_\ast\|^2.$$
\end{proof}

As a consequence of Theorem \ref{th:con}, when the restarts are carried out every $q$ reflection for $k$ times, i.e., total $kq$ iterations,
then  the convergence rate of restarted randomized surrounding method
is $\gamma^k=O(\frac{1}{q^k})$, which is much less that $O(\frac{1}{kq})$, the rate of the randomized surrounding  method if $k>2$ and $ q >2$.
It shows the potential that the restarted version
could be faster than the original one.

The restarted randomized surrounding method can be generalized into
more efficient approaches by introducing the relaxation.
For instance, the restarted randomized surrounding method can still converge when
the approximate solution
is chosen to be the convex linear combination $x_{k+1}=\sum_{i=0}^{q-1} \omega_{k}^{(i)}y_{k}^{(i)}$, where weights $\sum_{i=0}^{q-1}\omega_{k}^{(i)}=1$ and $\omega_{k}^{(i)}>0$.
Moreover, the restart number $q$ in the iteration
could be flexible to get  better numerical performances.

\section{Numerical experiments}

In this section, numerical experiments are presented to demonstrate the efficiency
of the restarted randomized surrounding method  and compared with the original randomized surrounding method.

All the methods start from the zero vector and stop when the norm of relative error vector (denoted by `ERR') satisfies
$$\text{ERR}=\frac{\|x_k-x_\ast\|_2^2}{\|x_0-x_\ast\|_2^2}\leq 10^{-6},$$
or achieves the maximal number of the iteration, e.g., 5000.
 The number of iteration steps (denoted by `IT'), the elapsed CPU time in seconds (denoted by `CPU')
of the randomized surrounding method (abbreviated as `RS') and the proposed restarted randomized surrounding method are compared.

{\bf Example 1.}
The test matrices are generated by using the Matlab function $A=\text{randn}(m,n)$
where the components are normally distributed random numbers.
The consistent linear system is constructed by $b=Ax_\ast$ where the exact solution $x_\ast$ is an all-one vector.

The curves of the relative error versus the number of the iteration are plotted in Figure \ref{fig:ex1} for the RS, RRS(5), RRS(10) and RRS(20) methods respectively.

\begin{figure}[htbp]
	\centering
	\subfigure[$A\in\mathbb{R}^{3000\times 100}$]{
	\centering
		\includegraphics[width=0.48\linewidth]{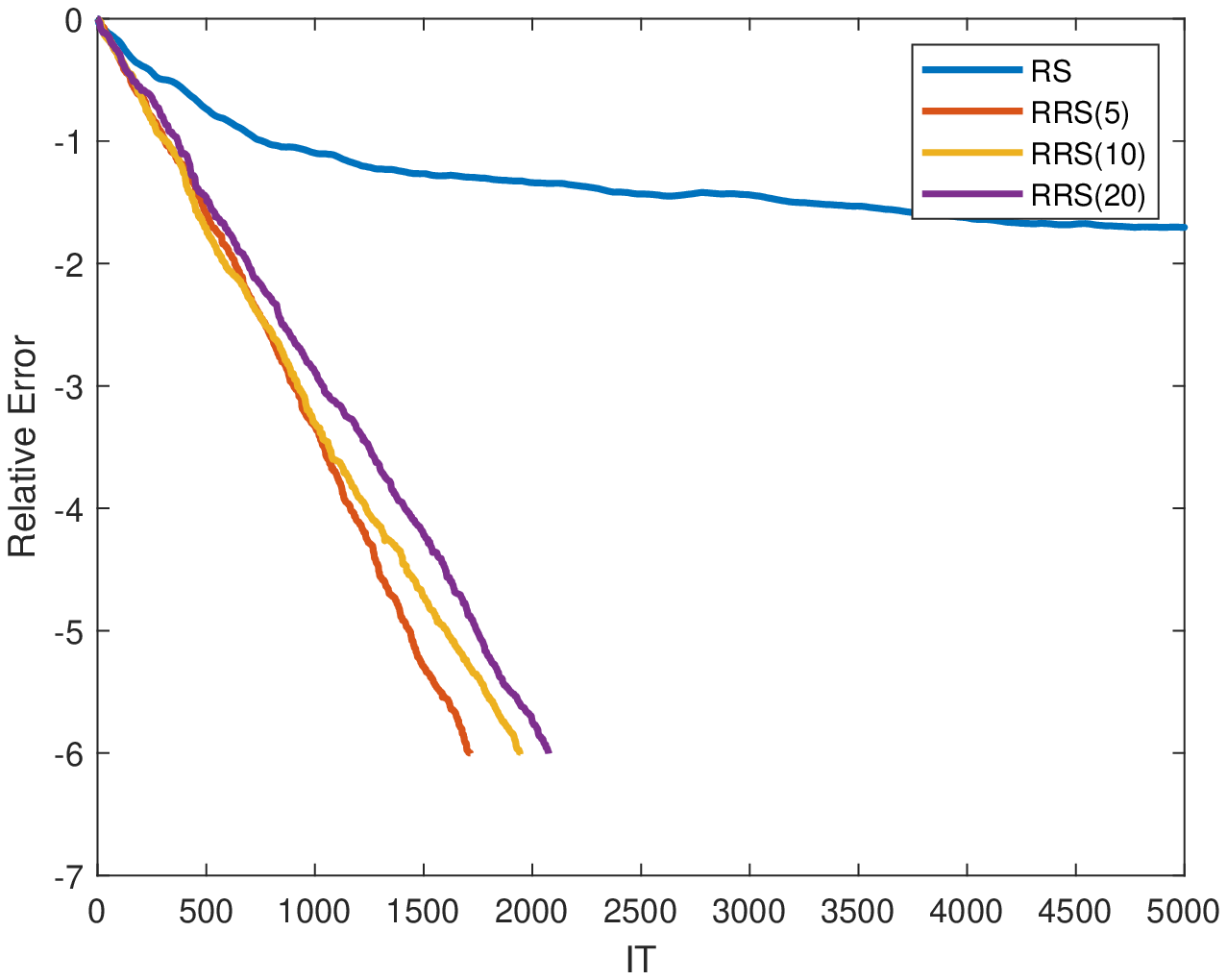}
	}
\subfigure[$A\in\mathbb{R}^{100\times 3000}$]{
		\centering
		\includegraphics[width=0.48\textwidth]{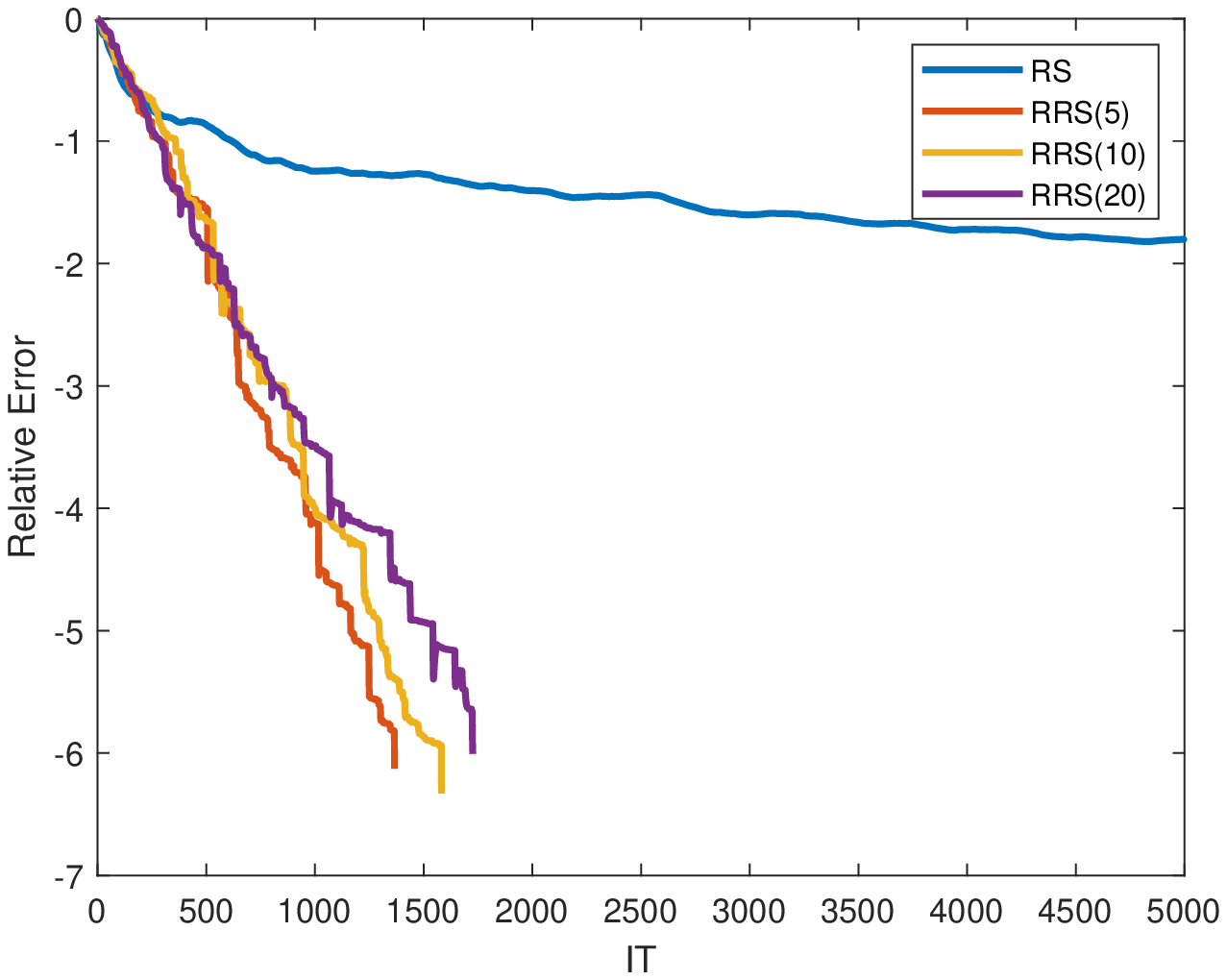}
}
\caption{Convergence curves for overdetermined and underdetermined cases}
\label{fig:ex1}
\end{figure}

From Figure \ref{fig:ex1}, it is observed that the restarted randomized surrounding method converges with
	expected linear rate and the curves decrease much steeper than that of randomized surrounding method  in both the underdetermined and overdetermined cases. It indicates that restarted randomized surrounding method is much efficient than the original randomized surrounding method for underdetermined and overdetermined cases.

In order to further compare the convergence performance, in Table \ref{tab:over} and Table \ref{tab:under}, the number of iteration and the elapsed CPU time of randomized surrounding method and restarted randomized surrounding method for $q=5,10$ and $20$ are listed for different sizes respectively.
All results are computed the average over 40 trials.

\begin{table}[!htbp]
	\begin{center}
		\begin{tabular}{|cc|c|c|c|c|c|} \hline
			\multicolumn{2}{|c|}{}  & $1000\times 100$ & $2000\times 100$ & $3000\times 100$ & $4000\times 100$ & $5000\times 100$ \\ \hline
			\multicolumn{1}{|c|}{\multirow{2}{*}{RS}}
			& IT   	  &  5000 	 & 5000 	 & 5000 	 &  5000 &	 5000  \\ \cline{2-7}
			\multicolumn{1}{|c|}{}& CPU  	 &  0.0088 	 & 0.0093 	 & 0.0100	 &  0.0097 &	 0.0107  \\ \hline
			\multicolumn{1}{|c|}{\multirow{2}{*}{RRS(5)}}& IT   	 &  1929 	 & 1830 	 & 1812 	 &  1804 &	 1776  \\ \cline{2-7}
			\multicolumn{1}{|c|}{}& CPU  	 &  0.0035 	 & 0.0036 	 & 0.0040	 &  0.0039 &	 0.0041  \\ \hline
			\multicolumn{1}{|c|}{\multirow{2}{*}{RRS(10)}} & IT   	 &  2062 	 & 1962 	 & 1952 	 &  1945 &	 1950  \\ \cline{2-7}
			\multicolumn{1}{|c|}{}& CPU  	 &  0.0036 	 & 0.0038 	 & 0.0050	 &  0.0041 &	 0.0044  \\ \hline
			\multicolumn{1}{|c|}{\multirow{2}{*}{RRS(20)}} & IT   	 &  2163 	 & 2092 	 & 2061 	 &  2064 &	 2043  \\ \cline{2-7}
			\multicolumn{1}{|c|}{}& CPU  	 &  0.0038 	 & 0.0041 	 & 0.0047	 &  0.0043 &	 0.0045  \\ \hline
		\end{tabular}
	\end{center}
	\caption{Numerical results for overdetermined cases}\label{tab:over}
\end{table}

\begin{table}[!htbp]
	\begin{center}
		\begin{tabular}{|cc|c|c|c|c|c|} \hline
			\multicolumn{2}{|c|}{}  & $100 \times 1000$ & $100\times 2000$ & $100 \times 3000$ & $100\times 4000$ & $100\times 5000$ \\ \hline
			\multicolumn{1}{|c|}{\multirow{2}{*}{RS}}
			& IT   	 &  5000 	 & 5000 	 & 5000 	 &  5000 &	 5000  \\ \cline{2-7}
			\multicolumn{1}{|c|}{}& CPU  	 &  0.0282 	 & 0.0584 	 & 0.1114	 &  0.1847 &	 0.3399  \\\hline
			\multicolumn{1}{|c|}{\multirow{2}{*}{RRS(5)}} & IT   	 &  1729 	 & 1608 	 & 1541 	 &  1531 &	 1472  \\ \cline{2-7}
			\multicolumn{1}{|c|}{}& CPU  	 &  0.0093 	 & 0.0182 	 & 0.0331	 &  0.0539 &	 0.0922  \\\hline
			\multicolumn{1}{|c|}{\multirow{2}{*}{RRS(10)}} & IT   	 &  1893 	 & 1740 	 & 1663 	 &  1672 &	 1666  \\ \cline{2-7}
			\multicolumn{1}{|c|}{}& CPU  	 &  0.0102 	 & 0.0194 	 & 0.0371	 &  0.0596 &	 0.1052  \\\hline
			\multicolumn{1}{|c|}{\multirow{2}{*}{RRS(20)}} & IT   	 &  1978 	 & 1893 	 & 1805 	 &  1775 &	 1741  \\ \cline{2-7}
			\multicolumn{1}{|c|}{}& CPU  	 &  0.0106 	 & 0.0211 	 & 0.0404	 &  0.0630 &	 0.1112  \\ \hline
		\end{tabular}
	\end{center}
	\caption{Numerical results for underdetermined cases}\label{tab:under}
\end{table}

From Table \ref{tab:over} and Table \ref{tab:under}, it is seen that
the restarted randomized surrounding methods are efficient,
and require less steps and CPU time than randomized surrounding method.
Among these methods, the restarted randomized surrounding methods with $q=5$  performs the best,
which indicates that a small number of restart may greatly improve the rate of convergence.

{\bf Example2.} The test matrices are chosen from the SuiteSparse Matrix Collection. The property of the matrices including size, density, rank and Euclidean condition number (i.e., cond ) of the tested matrices are given in Table \ref{info}.
	
\begin{table}[!htbp]
	\centering
	
	\begin{tabular}{|c|c|c|c|c|}
		\hline
		name                       & $crew1$          & $bibd\_13\_6$   & $cari$          & $bibd\_16\_8$          \\ \hline
		size                         & $135\times 6469$ & $78\times 1716$ & $400\times1200$ & $120\times 12870$  \\ \hline
		rank&135 &78&  400 &120\\\hline
		density                         & 5.38\%           & 19.23\%         & 31.83\%         & 23.33\%          \\ \hline
		cond                       & 18.20            & 6.27            & 3.13            & 9.54        \\ \hline
		
	\end{tabular}
	\caption{Information of the matrices from the  Matrix Market}\label{info}
\end{table}

	\begin{figure}[H]
		\centering
		\includegraphics[width=14cm]{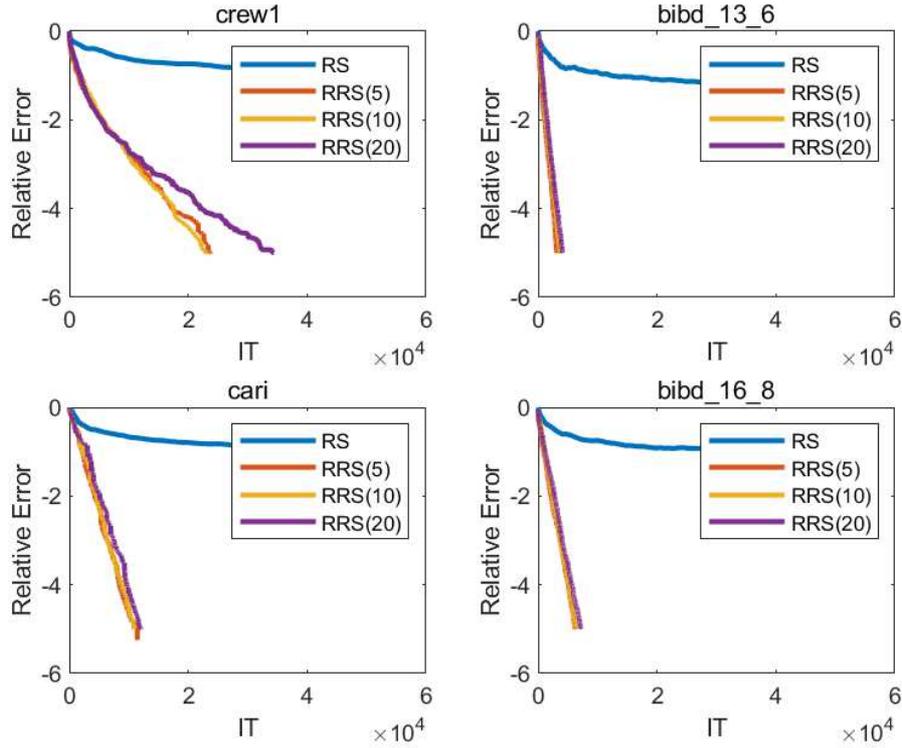}
		\caption{Convergence curves for the matrices from Matrix Market}\label{fig:mm}
	\end{figure}

The curves of the relative error versus the number of the iteration are plotted in Figure \ref{fig:mm} 	for randomized surrounding method and restarted randomized surrounding methods respectively. From Figure \ref{fig:mm}, it is seen that all the restarted
randomized surrounding methods converges faster than the randomized surrounding method, which further confirms the efficiency of the restart technique.

	In Table \ref{tab:mm}, the number of iteration and the elapsed CPU time of the randomized surrounding method and restarted randomized surrounding methods when $q=5,10,20$ are listed  respectively.

\begin{table}[!htbp]
	\begin{center}
		\begin{tabular}{|cc|c|c|c|c|} \hline
			\multicolumn{2}{|c|}{Method}  & $crew1$          & $bibd\_13\_6$   & $cari$          & $bibd\_16\_8$         \\ \hline
			\multicolumn{1}{|c|}{\multirow{2}{*}{RS}}
			& IT   	 &  5000 	 & 5000 	 & 5000 	 &  5000  \\ \cline{2-6}
			\multicolumn{1}{|c|}{}
			& CPU    &  22.3069   & 5.2526   & 4.2073  &  79.2887 \\\hline
			\multicolumn{1}{|c|}{\multirow{2}{*}{RRS(5)}}& IT     &  11921   &\bf{2027}   & \bf{6319}   &  \bf{3592} \\ \cline{2-6}
			\multicolumn{1}{|c|}{}   & CPU    &  5.2916   & \bf{0.2139}   & \bf{0.5388}  &  \bf{5.5906} \\\hline
			\multicolumn{1}{|c|}{\multirow{2}{*}{RRS(10)}} & IT     &  11545   & 2210   & 6901   &  3914\\ \cline{2-6}
			\multicolumn{1}{|c|}{}    & CPU    &  \bf{5.0950}   & 0.2319   & 0.5826  &  6.0782\\\hline
			\multicolumn{1}{|c|}{\multirow{2}{*}{RRS(20)}} & IT     &  \bf{11456}   & 2399   & 7256   &  4304 \\ \cline{2-6}
			\multicolumn{1}{|c|}{}    & CPU    &  5.1426   & 0.2502   & 0.6167  &  6.6959 \\\hline
		\end{tabular}
	\end{center}
	\caption{Numerical results for the matrices from Matrix Market}\label{tab:mm}
\end{table}

From Table \ref{tab:mm}, it is seen that the
restarted randomized surrounding methods are efficient,
and require less number of iteration and CPU time than the randomized surrounding method.
For the  matrix `crew1', RRS(10) require the least CPU time while RRS(20) take
the least number of iteration; for the other three examples,
RRS(5) require the least CPU time and the least number of iteration.
This implies that the restarted strategy is efficient and can greatly
improve the convergence while the optimal restart number is possibly  problem depended.

{\bf Example 3.} Finally, we compare these methods for solving a 2-D parallel-beam tomography problem and a seismic travel-time tomography problem generated by AIR tool box.
The right term is $b=Ax_\ast+e$ where $e$ is the noise vector and the relative noise level  is 0.01.
The signal-noise ratio (SNR) is defined as
$$\text{SNR}:=10\log_{10}\frac{\sum_{i=0}^{n}x_i^2}{\sum_{i=0}^{n}(x_i-\hat{x}_i)^2}$$
where $x$ is the original clean signal, $\hat{x}$ is the denoised signal, and $n$ is the length of the signal. The greater the value of SNR is, the better the denoising effect.
The reconstruction  images are compared  in Figures \ref{fig:pb} and \ref{fig:st} for  a 2-D parallel-beam tomography problem and a seismic travel-time tomography problem respectively, after $100m$ iterations where $m$ is the number of rows.

\begin{figure}[!htbp]
	\centering
	\includegraphics[width=14cm]{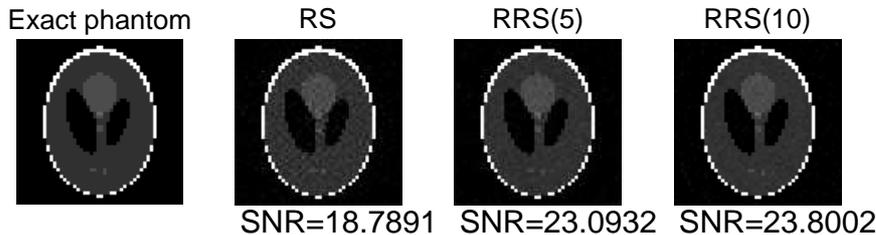}
	\caption{2-D parallel-beam tomography problem where $A\in\mathbb{R}^{23883\times 2500}$}\label{fig:pb}
\end{figure}

\begin{figure}[!htbp]
	\centering
	\includegraphics[width=14cm]{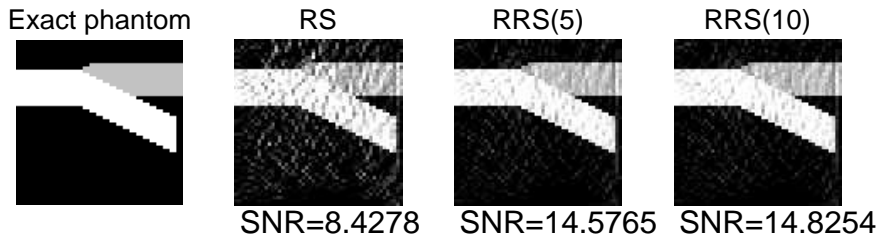}
	\caption{2-D seismic travel-time tomography problem where $A\in\mathbb{R}^{2500\times 2500}$}\label{fig:st}
\end{figure}

From Figures \ref{fig:pb} and \ref{fig:st}, it is observed that
the restarted randomized surrounding method can remove the noise and restore the real image efficiently.
Moreover,
the images reconstructed by the restarted randomized surrounding methods
are better than that of the randomized surrounding method
from the viewpoint of the sharpness of images and higher values of SNR.
Among the three approaches, it is seen that the recovered image of RRS(10) is the
best.

\section{Conclusions}\label{secconclu}

Restarted randomized surrounding methods are proposed for solving large linear problems. The convergence theory is established when the probability of row selection rule is proportional to the squared norm of row. Numerical experiments verify the proposed algorithms are efficient and outperform
the existing method for overdetermined and underdetermined linear equations, as well as in the application in image processing.
The continue work including of the dynamical restarted strategies, relaxation and the randomized selection rule are deserved to further study in the future.


\bibliographystyle{plain}

\end{document}